\documentclass[12pt]{article}

 \usepackage{graphics}
 \usepackage{graphicx}
 \usepackage{epsfig}
 \usepackage[all]{xy}

 \usepackage{amssymb}
 \usepackage{amsthm,amsmath}

 \usepackage{latexsym, epsfig, psfrag,eepic,colordvi,bm}
 \usepackage{graphics,color}
 \usepackage{amsmath,amsfonts,amssymb,amscd}
 \usepackage[all]{xy}
 \usepackage{epstopdf}

 \usepackage{amsthm,amsmath,amssymb}

 \usepackage{graphicx}
\usepackage{amsthm,amsmath,amssymb}

\usepackage{graphicx,epsfig}

\usepackage[colorlinks=true,citecolor=black,linkcolor=black,urlcolor=blue]{hyperref}
\usepackage{amsmath,blkarray}
\usepackage{arydshln}

\theoremstyle{plain}
\newtheorem{theorem}{Theorem}[section]
\newtheorem{lemma}[theorem]{Lemma}

\newtheorem{proposition}[theorem]{Proposition}

\theoremstyle{definition}
\newtheorem{definition}{Definition}[section]
\newtheorem{example}[definition]{Example}

\theoremstyle{remark}
\newtheorem{remark}[theorem]{Remark}

\date{}
\title{Linear sequential dynamical systems, incidence algebras, and M\"{o}bius functions}

\author{Ricky X. F. Chen$^1$, Christian M. Reidys$^2$\\
	\small Virginia Bioinformatics Institute and Dept. of Mathematics,\\[-0.8ex]
	\small Virginia Tech, 1015 Life Science Circle,\\[-0.8ex]
	\small Blacksburg, VA 24061, USA\\
	\small\tt $^1$chen.ricky1982@gmail.com, $^2$duck@santafe.edu
}

\begin{document}
\maketitle
\noindent{\bf Abstract.}
A sequential dynamical system (SDS) consists of a graph, a set of local functions
and an update schedule. A linear sequential dynamical system is an SDS whose local
functions are linear. In this paper, we derive an explicit closed formula for any linear SDS
as a synchronous dynamical system. We also show
constructively, that any synchronous linear system can be expressed as a
linear SDS, i.e.~it can be written as a product of linear local functions. 
Furthermore, we study the connection between linear SDS and the incidence algebras of
partially ordered sets (posets). Specifically, we show that the M\"{o}bius function of any
poset can be computed via an SDS, whose graph is induced by the Hasse diagram
of the poset.
Finally, we prove a cut theorem for the M\"{o}bius functions of posets with respect to certain chain
decompositions.

\noindent{\bf Keywords.}
	sequential dynamical system; linear SDS; partially ordered set; incidence algebra;
	M\"{o}bius function; homogeneous chain decomposition
	
%	\MSC 37E15 \sep 06A11 \sep 15A23
	
	
	
%\end{keyword}

\section{Introduction}
\label{sec1}
Modeling discrete processes is an important topic in mathematics as it is of relevance
for the understanding of network dynamics~\cite{rei1,rei2,rei3,rei4,laub,BE,kauf,vonn,reg,wolf}. Sequential dynamical systems
(SDS\footnote{We will write SDS in singular as well as plural form.})
\cite{rei1,rei2,rei3,rei4,rei5} provide such a framework.
A \emph{sequential dynamical system} over a graph involves three ingredients: a (dependency) graph to
represent the interaction relations among entities, local functions and an update schedule in order
to specify the order of state updates. SDS are motivated by the work on Cellular Automata
(CA)~\cite{vonn,wolf} and Boolean Networks~\cite{kauf}. The key question is to deduce
phase space properties from graph, local maps and update schedule.

It is not easy to deduce phase space properties of SDS over general graphs and general local functions
with varying update schedules. The question is intriguing though, since sometimes dramatic changes
can be induced by ``slightly'' changing the update schedule, while keeping everything else fixed.
In praxis, such changes can inadvertently be introduced when implementing an SDS on a specific
computer architecture.

In \cite{aledo1,rei1,rei2,rei3,laub,bill,rchen,defant,laub2} the reader can find a plethora of results on
dynamical systems over particular graphs and updating functions.
Linear systems together with a parallel update schedule, i.e.~dynamical systems, in which the
local functions are linear, are studied in~\cite{BE,tol,laub3}.

In~\cite{bill}, {linear SDS} are studied and a formal composition formula of matrices is obtained,
characterizing a linear SDS. We improve this result by deriving a simple closed matrix expression
for any linear SDS, thereby expressing a linear SDS as a synchronous system.
We provide an explicit construction expressing any synchronous linear system as an SDS. That is
using the LU-decomposition of a matrix, we construct a graph together with a family of linear local
functions. This construction is of interest in the context of the general question whether or not
a synchronous dynamical system, say $\mathbb{F}_2^n\rightarrow \mathbb{F}_2^n$, can be expressed as an SDS. For a general synchronous dynamical system, it is generally not easy to find a non-trivial SDS expression.

We then extend our framework to linear SDS over words, i.e.~update schedules in which vertices can
appear multiple times but at least once. We establish a closed matrix formula for these systems
and derive a relation to the M\"{o}bius functions of partially ordered sets associated with composing
such words.

The paper is organized as follows: in Section~\ref{sec2}, we review some basic concepts and facts on SDS.
In Section~\ref{sec3}, we study linear SDS with permutation update schedules and compute in Theorem~\ref{2thm1}
a matrix that completely describes such a system.
Using Theorem~\ref{2thm1}, we prove constructively in Theorem~\ref{3thm2} that any synchronous
linear system has a linear SDS equivalent. 
We also present an approach to compute the M\"{o}bius function of a poset with an SDS.
In Section~\ref{sec4}, we extend our results to linear SDS, where the update schedules are words over
all vertices (linear SDS over words). We compute in Theorem~\ref{4thm1} an explicit closed formula for
these, that restricts to Theorem~\ref{2thm1} proved in Section~\ref{sec3}. The concatenation of linear SDS
over words gives rise to a relation of the M\"{o}bius functions of posets, which allows us 
to prove a cut theorem in Theorem~\ref{T:rel} w.r.t. homogeneous chain decompositions of posets.
In Section~\ref{sec5}, we summarize the results of this paper, and discuss extending the framework to block-sequential dynamical systems as well as future work.

%%%
%%%%%%%%%%%%%%%%%%%%%%%%%%%%%%%%%%%%%%%%%%%%%%%%%%%%%%%%%%%%%%%%%%%%%%%%%%%%%%%%%%%%%%%%%%%%%%%%%%%%%%%%%%%%%%%%%%%%
%%%
\section{Linear SDS and the incidence algebras of posets}\label{sec2}
%%%
%%%%%%%%%%%%%%%%%%%%%%%%%%%%%%%%%%%%%%%%%%%%%%%%%%%%%%%%%%%%%%%%%%%%%%%%%%%%%%%%%%%%%%%%%%%%%%%%%%%%%%%%%%%%%%%%%%%%
%%%

Let $G$ be a connected, simple graph with vertex set $V(G)=\{v_1,\dots,v_n\}$. To each
vertex $v_i$, we associate a state $x_i\in \mathbb{K}$, where $\mathbb{K}$ is a finite field.
A function $f_{v_i}$ updates the state of $v_i$ based on the states of its neighbors and
$x_i$ itself.
Let $\pi=\pi_1\cdots\pi_m$ be a sequence of elements in $V(G)$, in which each $v_i$ appears at least
once, called the \emph{update schedule}.
The states of all vertices are sequentially updated according to their
order in $\pi$ left-to-right, i.e.,
the vertex specified by $\pi_{i+1}$ is updated after the vertex specified by $\pi_i$. Iterating $\pi$
generates the SDS, $(G,f,\pi)$, having $G$ as its \emph{dependency graph}, $f_{v_i}$ ($1\leq i \leq n$)
as \emph{local functions}, and $\pi$ as the \emph{update schedule}.

\begin{remark}
	Alternatively, an SDS can be defined directly via its local functions and update schedule,
	the former inducing the underlying graph by their dependencies.
\end{remark}

The vectors $X=[x_1,\dots, x_n]^T\in \mathbb{K}^n$ are called the \emph{system states} of the SDS.
An update will produce the \emph{phase space} mapping $X_0$ into $X_1=(G,f,\pi)(X_0)$. Understanding the time
evolution of the SDS is tantamount to understanding the mapping $(G, f, \pi): \mathbb{K}^n \rightarrow
\mathbb{K}^n $, i.e.~the forest of directed unicyclic graphs it induces.
In the case of the system map $(G,f,\pi)$ being a linear transformation on $\mathbb{K}^n$, the phase space
has been studied in \cite{BE,tol,laub3,reg} and phase space features were described in terms of the
annihilating polynomials and invariant subspaces of the linear transformation. In~\cite{bill}, linear
SDS over directed graphs were analyzed, allowing for states $x_i$ contained in a finite algebra.

%%%
%%%%%%%%%%%%%%%%%%%%%%%%%%%%%%%%%%%%%%%%%%%%%%%%%%%%%%%%%%%%%%%%%%%%%%%%%%%%%%%%%%%%%%%%%%%%%%%%%%%%%%%%%%%%%%%%%%%%
%%%
\begin{definition}
	A \emph{linear} {SDS} over $\mathbb{K}$ is an SDS $(G, f, \pi)$ where the local functions are defined by
	\begin{align}
	f_{v_i}: x_i\mapsto a_{i1}x_1+a_{i2}x_2+\dots+a_{in}x_n,
	\end{align}
	where $a_{ij}\in \mathbb{K}$ and $a_{ij}=0$ if $v_i$ and $v_j$ are not adjacent in $G$.
\end{definition}
%%%
%%%%%%%%%%%%%%%%%%%%%%%%%%%%%%%%%%%%%%%%%%%%%%%%%%%%%%%%%%%%%%%%%%%%%%%%%%%%%%%%%%%%%%%%%%%%%%%%%%%%%%%%%%%%%%%%%%%%
%%%

Note that each local function $f_{v_i}$ induces the function $F_{v_i}$:
$$
F_{v_i}: \mathbb{K}^n \rightarrow \mathbb{K}^n, \quad [x_1,\dots, x_i,\dots, x_n]\mapsto
[x_1,\dots, f_{v_i}(x_i), \dots, x_n],
$$
and $F_{v_i}$ can be expressed as a matrix over $\mathbb{K}$ as follows:
\begin{align*}
F_{v_i}=\left(\begin{array}{ccccc}
1 & 0 &0&\cdots &0\\
0 &1  &0&\cdots &0\\
\vdots & \vdots & \ddots & &\vdots\\
a_{i1}&a_{i2}&a_{i3}&\cdots&a_{in}\\
\vdots & \vdots & \vdots & \ddots&\vdots\\
0&0& 0 &\cdots &1
\end{array}\right), \quad
\text{i.e.,}\ [F_{v_i}]_{kj}=\begin{cases}
1, & \text{if}\ (k=j)\wedge (k\neq i);\\
a_{ij}, & \text{if}\ k=i;\\
0, & \text{otherwise}.
\end{cases}
\end{align*}
Then, if $X_1=(G,f,\pi)(X_0)$, we have
\begin{align}
X_1=F_{\pi_n}F_{\pi_{n-1}}\cdots F_{\pi_1}(X_0).
\end{align}

The first objective in Section~\ref{sec3} will be to derive a closed formula for the composition of these matrices
$F_{\pi_i}$. Secondly, given a linear transformation $A$, we will prove constructively that there exists a linear SDS such that $A=(G,f,\pi)$.

It is well known that an update schedule $\pi$, being a permutation on $V(G)$, induces an \emph{acyclic orientation} $Acyc(G,\pi)$ of the dependency graph
$G$~\cite{rei3,rei4,rei6}, i.e., the edge between two adjacent vertices $v_i$ and $v_j$ will be oriented
from $v_i$ to $v_j$ iff $v_i$ appears after $v_j$ in $\pi$. It is
obvious that
$(G,f,\pi_1)=(G,f,\pi_2)$ if $Acyc(G,\pi_1)=Acyc(G,\pi_2)$.

Any acyclic orientation of $G$ determines a partial order among the $G$-vertices, i.e., a poset.
Let us recall some basic facts about posets~\cite[Chapter~$3$]{stan1}.

A \emph{poset} is a set $P$ with a binary relation `$\leq$' among the elements in $P$, where
the binary relation satisfies reflexivity, antisymmetry and transitivity. The poset will
be denoted by $(P,\leq)$, or $P$ for short.
If two elements $x$ and $y$ satisfy $x\leq y$, we say $x$ and $y$ are \emph{ comparable}.
Furthermore, if $x \leq y$ but $x \neq y$, we write $x < y$.
A \emph{linear extension} of a poset $P$ is an arrangement of the elements in $P$, say
$s_{i_1}s_{i_2}\cdots s_{i_n}$, such that $j\leq k$ if $s_{i_j}\leq s_{i_k}$ in $P$.

%%%
%%%%%%%%%%%%%%%%%%%%%%%%%%%%%%%%%%%%%%%%%%%%%%%%%%%%%%%%%%%%%%%%%%%%%%%%%%%%%%%%%%%%%%%%%%%%%%%%%%%%%%%%%%%%%%%%%%%%
%%%
\begin{definition}
	Given a poset $(P,\leq)$, a (closed) \emph{interval} $[x,y]$, where $x\leq y$, is the subset
	$\{z\in P: x\leq z \leq y\}$. A \emph{chain} of $P$ is a subset of elements where any
	two elements are comparable.
\end{definition}
%%%
%%%%%%%%%%%%%%%%%%%%%%%%%%%%%%%%%%%%%%%%%%%%%%%%%%%%%%%%%%%%%%%%%%%%%%%%%%%%%%%%%%%%%%%%%%%%%%%%%%%%%%%%%%%%%%%%%%%%
%%%
In this paper, we only consider posets $(P, \leq)$, where the cardinality $|P|$ is finite and we denote $Int(P)$
to be the set of all intervals of the poset $(P, \leq)$. For a function $h: Int(P)\rightarrow
\mathbb{K}$, we write $h(x,y)$ instead of $h([x,y])$.

%%%
%%%%%%%%%%%%%%%%%%%%%%%%%%%%%%%%%%%%%%%%%%%%%%%%%%%%%%%%%%%%%%%%%%%%%%%%%%%%%%%%%%%%%%%%%%%%%%%%%%%%%%%%%%%%%%%%%%%%
%%%
\begin{definition}
	The \emph{incidence algebra} $I(P, \mathbb{K})$ of $(P, \leq)$ over $\mathbb{K}$ is the $\mathbb{K}$-algebra
	of all functions
	$
	h: Int(P)\rightarrow \mathbb{K}
	$,
	where multiplication of two functions $h$ and $r$ is defined by
	$$
	(h\cdot r)(x,y):=\sum_{x\leq z \leq y} h(x,z)r(z,y),
	$$
	and $(h\cdot r\cdot g)(x,y)=((h\cdot r)\cdot g)(x,y)$.
\end{definition}
%%%
%%%%%%%%%%%%%%%%%%%%%%%%%%%%%%%%%%%%%%%%%%%%%%%%%%%%%%%%%%%%%%%%%%%%%%%%%%%%%%%%%%%%%%%%%%%%%%%%%%%%%%%%%%%%%%%%%%%%
%%%

See Stanley~\cite[Section~$3.6,\ 3.7$]{stan1} for more on incidence algebras of posets. 

Let $P=\{s_1,\ldots, s_n\}$. The function $h$ can be expressed as an $n\times n$ matrix $H$ where the rows
(and columns) of $H$ are indexed sequentially by $s_1,\dots, s_n$ and the $(i,j)$-th entry is
$$
[H]_{ij}=
\begin{cases}
h(s_i,s_j), & \text{if}\ s_i\leq s_j; \\
0, & \text{otherwise.}
\end{cases}
$$
It is easy to check that the corresponding matrix of $h\cdot r$ equals the product of the matrices $H$ and $R$, where $R$ is the corresponding matrix
of the function $r$. If $HR=RH=I$, $H$ ($h$) and $R$ ($r$) are inverse to each other, i.e.~we have $H^{-1}=R$
and $R^{-1}=H$. Furthermore, let ${\rm diag}\{a_1, \dots, a_n\}$ denote the $n \times n$ diagonal matrix where the $(i,i)$-th entry is $a_i$. Then, we have the following lemma.
%%%
%%%%%%%%%%%%%%%%%%%%%%%%%%%%%%%%%%%%%%%%%%%%%%%%%%%%%%%%%%%%%%%%%%%%%%%%%%%%%%%%%%%%%%%%%%%%%%%%%%%%%%%
%%%
\begin{lemma}\label{2lem1} For $k \geq 1$, we have
	\begin{align}
	[H^k]_{ij} & =\sum_{s_i\leq s_{l_1} \leq s_{l_2} \leq \cdots \leq s_{l_{k-1}} \leq s_j} h(s_i,s_{l_1})h(s_{l_1},s_{l_2})
	\cdots h(s_{l_{k-1}},s_j),\\
	[(H-{\rm Diag}\{H\})^k]_{ij} & =\sum_{s_i< s_{l_1} < s_{l_2} < \cdots < s_{l_{k-1}} < s_j} h(s_i,s_{l_1})h(s_{l_1},s_{l_2})
	\cdots h(s_{l_{k-1}},s_j),
	\end{align}
	where ${\rm Diag}\{H\}={\rm diag}\{[H]_{11},\dots,[H]_{nn}\}$.
\end{lemma}
%%%
%%%%%%%%%%%%%%%%%%%%%%%%%%%%%%%%%%%%%%%%%%%%%%%%%%%%%%%%%%%%%%%%%%%%%%%%%%%%%%%%%%%%%%%%%%%%%%%%%%%%%%%
%%%
\proof
The first equation follows from the multiplication rule of matrices and the fact that $[H]_{kl}=0$ if $s_k\nleq s_l$. The second equation follows by further noticing that $[H-{\rm Diag}\{H\}]_{kl}=0$ if $s_k \nless s_l$. \qed

%%%
%%%%%%%%%%%%%%%%%%%%%%%%%%%%%%%%%%%%%%%%%%%%%%%%%%%%%%%%%%%%%%%%%%%%%%%%%%%%%%%%%%%%%%%%%%%%%%%%%%%%%%%
%%%
\section{Linear SDS over permutations and M\"{o}bius functions}\label{sec3}
%%%
%%%%%%%%%%%%%%%%%%%%%%%%%%%%%%%%%%%%%%%%%%%%%%%%%%%%%%%%%%%%%%%%%%%%%%%%%%%%%%%%%%%%%%%%%%%%%%%%%%%%%%%
%%%

Let $(G,f,\pi)$ be a linear SDS. Then its local functions $f_{v_i}$ can be encoded via the matrix
\begin{align}
A=	\begin{blockarray}{cccccc}
v_1& v_2 & v_3 & \cdots & v_n &\\
\begin{block}{(ccccc)c}
a_{11} & a_{12} &a_{13}&\cdots &a_{1n} & v_1\\
a_{21} &a_{22}  &a_{23}&\cdots &a_{2n} & v_2\\
\vdots & \vdots & \vdots & &\vdots &\vdots \\
a_{i1}&a_{i2}&a_{i3}&\cdots&a_{in} & v_i\\
\vdots & \vdots & \vdots & &\vdots & \vdots\\
a_{n1}&a_{n2}& a_{n3} & \cdots &a_{nn} & v_n\\
\end{block}
\end{blockarray}
\end{align}
and we set $(G,A,\pi)=(G,f,\pi)$. Note that if all vertices update simultaneously, i.e.,~we have a synchronous (or parallel)
dynamical system, the system map is the linear transformation $A$. 

The acyclic orientation $Acyc(G,\pi)$ induces a poset $(P, \leq)$ as follows:
\begin{itemize}
	\item[(a)] $P=V(G)$;
	\item[(b)] $v_i \leq v_j$ if $v_i=v_j$ or there is an edge oriented from $v_i$ to $v_j$;
	\item[(c)] the binary relation $\leq$ on $P$ is the transitive closure w.r.t.~(b).
\end{itemize}
Meanwhile, the local functions $A$ of the SDS induce a function $h_{\pi} \in I(P, \mathbb{K})$:
$$
h_{\pi}: Int(P)\rightarrow \mathbb{K}, \quad [v_i,v_j] \mapsto a_{ij}.
$$
Let $H_{\pi}$ be the matrix of $h_{\pi}$ and $A_{\pi}=H_{\pi}-{\rm Diag}\{H_{\pi}\}=H_{\pi}-{\rm Diag}\{A\}$.
Equivalently, $A_{\pi}$ can be constructed from the matrix $A$ by
\begin{itemize}
	\item[(i)] setting all diagonal entries to $0$, and 
	\item[(ii)] for each pair $a_{ij}$ and $a_{ji}$, $i\neq j$, set one of these equal to $0$. Specifically, set $a_{ij}=0$ if $v_i$ precedes $v_j$ in $\pi$. 
\end{itemize}

\begin{example}\label{eg31}
	Consider the graph $G=C_4$ as shown in Figure~\ref{2fig1} (left).
	\begin{figure}[!htb]
		\centering
		\includegraphics[scale=.75]{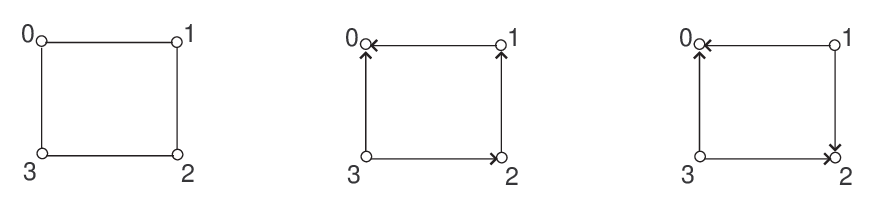}
		\caption{The graph $G=C_4$ (left), the acyclic orientation $Acyc(G,0123)$ (middle), and the acyclic orientation $Acyc(G,0213)$ (right).}\label{2fig1}
	\end{figure}
	
	\noindent Over $\mathbb{F}_2$, let the local functions 
	$$
	f_i: x_i\mapsto x_{i-1}+x_{i}+x_{i+1} \pmod 2,
	$$
	where the indices are taken modulo $4$. Then, we have
	
	\begin{align*}
	A=\left(\begin{array}{cccc}
	1&1&0&1\\
	1&1&1&0\\
	0&1&1&1\\
	1&0&1&1
	\end{array}\right),\quad
	A_{0123}=\left(\begin{array}{cccc}
	0&0&0&0\\
	1&0&0&0\\
	0&1&0&0\\
	1&0&1&0
	\end{array}\right),\quad
	A_{0213}=\left(\begin{array}{cccc}
	0&0&0&0\\
	1&0&1&0\\
	0&0&0&0\\
	1&0&1&0
	\end{array}\right).
	\end{align*}
\end{example}

We are now in position to present our first result:

%%%
%%%%%%%%%%%%%%%%%%%%%%%%%%%%%%%%%%%%%%%%%%%%%%%%%%%%%%%%%%%%%%%%%%%%%%%%%%%%%%%%%%%%%%%%%%%%%%%%%%%
%%%
\begin{theorem}\label{2thm1}
	Let $(G,A,\pi)$ be a linear SDS. Then,
	\begin{align}\label{1t1}
	(G,A,\pi)=  (I-A_{\pi})^{-1}(A-A_{\pi}).
	\end{align}
\end{theorem}
%%%
%%%%%%%%%%%%%%%%%%%%%%%%%%%%%%%%%%%%%%%%%%%%%%%%%%%%%%%%%%%%%%%%%%%%%%%%%%%%%%%%%%%%%%%%%%%%%%%%%%%
%%%

\begin{proof}
	First, note that $F_{v_i}$ can be decomposed into $F_{v_i}=I+\bar{F}_{v_i}$, where $\bar{F}_{v_i}$ has
	nonzero entries only at the $i$-th row as follows:
	\begin{align*}
	\bar{F}_{v_i}=\left(\begin{array}{ccccc}
	0 & 0 &0&\cdots &0\\
	0 &0  &0&\cdots &0\\
	\vdots & \vdots & \vdots & &\vdots\\
	a_{i1}&\cdots &a_{ii}-1&\cdots&a_{in}\\
	\vdots & \vdots & \vdots & &\vdots\\
	0&0&  0 &\cdots &0
	\end{array}\right).
	\end{align*}
	Accordingly,
	\begin{align}
	(G,A,\pi)=F_{\pi}=F_{\pi_n}F_{\pi_{n-1}}\cdots F_{\pi_1}&=(I+\bar{F}_{\pi_n})(I+\bar{F}_{\pi_{n-1}})\cdots
	(I+\bar{F}_{\pi_1})\nonumber \\ 
	&=I+\sum_{m=1}^{n}\sum_{n\geq k_m>\cdots>k_1\geq 1}\bar{F}_{\pi_{k_m}}\cdots\bar{F}_{\pi_{k_1}}.
	\end{align}
	We shall consider the value of the $(i,j)$-th entry $[F_{\pi}]_{ij}$ of $F_{\pi}$. Let $s=s_1\cdots s_k$
	be a sequence. If $s'=s_{i_1}\cdots s_{i_j}$ is a subsequence of $s$,
	we write $s'\lhd s$.
	
	{\it Claim~$1$.} For $1\leq m,i,j \leq n$, we have
	\begin{align}\label{2thm1eq1}
	\left[\sum_{n\geq k_m>\cdots>k_1\geq 1}\bar{F}_{\pi_{k_m}}\cdots\bar{F}_{\pi_{k_1}}\right]_{ij}
	&=\sum_{v_{l_{m-1}}v_{l_{m-2}}\cdots v_{l_1}v_i \lhd \pi} [\bar{F}_{v_{i}}]_{i l_{1}}[\bar{F}_{v_{l_{1}}}]_{l_{1} l_{2}}\cdots
	[\bar{F}_{v_{l_{m-1}}}]_{l_{m-1} j}.
	\end{align}

	To prove the claim, we observe
	\begin{align*}
	\left[\sum_{n\geq k_m>\cdots>k_1\geq 1}\bar{F}_{\pi_{k_m}}\cdots\bar{F}_{\pi_{k_1}}\right]_{ij}=
	\sum_{n\geq k_m>\cdots>k_1\geq 1} \sum_{\substack{1\leq l_t\leq n\\ 1\leq t\leq m-1}}[\bar{F}_{\pi_{k_m}}]_{il_1}[\bar{F}_{\pi_{k_{m-1}}}]_{l_1l_2}\cdots [\bar{F}_{\pi_{k_2}}]_{l_{m-2}l_{m-1}}[\bar{F}_{\pi_{k_1}}]_{l_{m-1}j}.
	\end{align*}
	Note, if $\pi_{k_m} \neq v_i$, then $[\bar{F}_{\pi_{k_m}}]_{il_1}=0$. Similarly, for $1 \leq t \leq m-1$, if
	$\pi_{k_{m-t}} \neq v_{l_t}$, then $[\bar{F}_{\pi_{k_{m-t}}}]_{l_{t}l_{t+1}}=0$. Thus, the RHS
	can be reduced to
	\begin{align*}
	\sum_{v_{l_{m-1}}v_{l_{m-2}}\dots v_{l_1}v_i \lhd \pi}[\bar{F}_{v_{i}}]_{i l_{1}}[\bar{F}_{v_{l_{1}}}]_{l_{1} l_{2}}\cdots
	[\bar{F}_{v_{l_{m-2}}}]_{l_{m-2} l_{m-1}}[\bar{F}_{v_{l_{m-1}}}]_{l_{m-1} j},
	\end{align*}
	whence Claim~$1$.
	
	{\it Claim~$2$.} For $1\leq m, i, j \leq n$, we have
	\begin{align}\label{2thm1eq2}
	\left[\sum_{n\geq k_m>\cdots>k_1\geq 1}\bar{F}_{\pi_{k_m}}\cdots\bar{F}_{\pi_{k_1}}\right]_{ij}=[A_{\pi}^{m-1}(A-I)]_{ij}.
	\end{align}
	
	Note if $v_i$ and $v_{l_1}$ are adjacent in $G$, $v_i$ appearing after $v_{l_1}$ implies
	$v_i < v_{l_1}$ in the induced poset. Thus, $[\bar{F}_{v_{i}}]_{i l_{1}}$ is either equal to
	$h_{\pi}(v_i,v_{l_1})$ or $0$ depending on $v_i$ and $v_{l_1}$ being adjacent or not.
	The same holds for $[\bar{F}_{v_{l_{t}}}]_{l_t l_{t+1}}$ for all $1\leq t \leq m-2$.
	Therefore, the RHS of eq.~\eqref{2thm1eq1} is
	\begin{align*}
	&\sum_{1\leq l_{m-1} \leq n}\Bigg(\sum_{v_{i}< v_{l_{1}}<\cdots < v_{l_{m-1}}}h_{\pi}(v_{i},v_{l_{1}})h_{\pi}(v_{l_{1}},v_{l_{2}})\cdots h_{\pi}(v_{l_{m-2}},v_{l_{m-1}})\Bigg)
	[\bar{F}_{v_{l_{m-1}}}]_{l_{m-1} j}\\
	= &\sum_{1\leq l_{m-1} \leq n} [A_{\pi}^{m-1}]_{il_{m-1}}[\bar{F}_{v_{l_{m-1}}}]_{l_{m-1} j},
	\end{align*}
	where the last equality follows from Lemma~\ref{2lem1}, namely for fixed $i$ and $l_{m-1}$,
	\begin{align*}
	\sum_{v_{i}< v_{l_{1}}<\cdots < v_{l_{m-1}}}h(v_{i},v_{l_{1}})h(v_{l_{1}},v_{l_{2}})\cdots h(v_{l_{m-2}},v_{l_{m-1}})=[A_{\pi}^{m-1}]_{il_{m-1}}.
	\end{align*}
	Consequently, we have
	\begin{align*}
	\left[\sum_{n\geq k_m>\cdots>k_1\geq 1}\bar{F}_{\pi_{k_m}}\cdots\bar{F}_{\pi_{k_1}}\right]_{ij}
	&=\sum_{1\leq l_{m-1} \leq n}[A_{\pi}^{m-1}]_{il_{m-1}}[\bar{F}_{v_{l_{m-1}}}]_{l_{m-1} j}\\
	&=\sum_{1\leq l_{m-1} \leq n}[A_{\pi}^{m-1}]_{il_{m-1}}[A-I]_{l_{m-1} j}
	=[A_{\pi}^{m-1}(A-I)]_{ij},
	\end{align*}
	completing the proof of Claim~$2$.
	
	Summing over all $1\leq m \leq n$, we obtain
	\begin{align*}
	F_{\pi}&=I+(A-I)+A_{\pi}(A-I)+A_{\pi}^2(A-I)+\cdots+A_{\pi}^{n-1}(A-I)\\
	&=I+(I+A_{\pi}+A_{\pi}^2+\cdots+A_{\pi}^{n-1})(A-I).
	\end{align*}
	Finally, we observe that
	\begin{align}
	(I-A_{\pi})(I+A_{\pi}+A_{\pi}^2+\cdots+A_{\pi}^{n-1})=I-A_{\pi}^n.
	\end{align}
	According to Lemma~\ref{2lem1}, $A_{\pi}^n=0$ since there is no chain
	of size larger than $n$ in a poset with $n$ elements.
	Thus, we obtain $I+A_{\pi}+A_{\pi}^2+\cdots+A_{\pi}^{n-1}=(I-A_{\pi})^{-1}$.
	Hence,
	$$
	F_{\pi}=I+(I-A_{\pi})^{-1}(A-I)=(I-A_{\pi})^{-1}(A-A_{\pi}),
	$$
	completing the proof of the theorem.
\end{proof}

%%%
%%%%%%%%%%%%%%%%%%%%%%%%%%%%%%%%%%%%%%%%%%%%%%%%%%%%%%%%%%%%%%%%%%%%%%%%%%%%%%%%%%%%%
%%%
\begin{theorem}\label{3thm2}
	For a finite field $\mathbb{K}$, any synchronous linear dynamical system map $\mathbb{K}^n\rightarrow \mathbb{K}^n$ can be expressed as a linear SDS over $\mathbb{K}$.
\end{theorem}
%%%
%%%%%%%%%%%%%%%%%%%%%%%%%%%%%%%%%%%%%%%%%%%%%%%%%%%%%%%%%%%%%%%%%%%%%%%%%%%%%%%%%%%%%
%%%
\proof It is well known that any square matrix $T$ has an $\text{LUP}$ decomposition~\cite{turing,TI}, i.e., there exists a
permutation matrix $P$, a unit lower triangular matrix $L$ (i.e.,~all diagonal entries are $1$ in $L$), and an upper triangular matrix $U$ such that
$PT=LU$. By construction the matrix $PT$ is obtained from $T$ by reordering the rows.
Recall that a synchronous linear dynamical system is a linear system with parallel updating and the system map
is described by a matrix whose rows encode the local functions.
Clearly, different orderings of the local functions (into rows) can be easily converted into each other,
thus has no effect on understanding the system dynamics.
Accordingly, we may assume that the matrix $T$ of a linear synchronous system has w.l.o.g.~the property that $T=LU$.
We further assume the rows and columns of $T$ are indexed by $v_1,\dots, v_n$. 
Since $L$ is unit lower triangular, it is invertible and
its inverse is also a unit lower triangular matrix. As a result, there exists a lower triangular matrix, $A$, having $0$ as
diagonal entries, such that $(I-A)^{-1}=L$.

{\it Claim.} The linear transformation $T$ equals the system map of the linear SDS $(G,U+A,\pi)$, where
\begin{description}
	\item[(a)] the dependency graph $G$ is determined by $U+A$, where if either $[U+A]_{ij}\neq 0$ or  $[U+A]_{ji}\neq 0$ , then
	$v_i$ and $v_j$ are adjacent in $G$, and not adjacent otherwise,
	\item[(b)] the update schedule $\pi=v_1\cdots v_n$.
\end{description}
The particular choice of $\pi$ implies that $(U+A)_{\pi}$ is a lower triangular matrix, i.e.,~any entry above the main diagonal is set to zero.
Since $U$ is upper triangular and $A$ is lower triangular, we have $(U+A)_\pi=A$.
Then, by Theorem~\ref{2thm1} we arrive at
$$
(G,(U+A),\pi)=(I-A)^{-1}((U+A)-A)=LU=T,
$$
completing the proof of the theorem.
\qed

An SDS is called \emph{invertible} if the map $(G,f,\pi): \mathbb{K}^n \rightarrow \mathbb{K}^n$ is invertible. Let $\pi^{[r]}$ denote the reverse order of $\pi$, i.e., $\pi^{[r]}=\pi_n\cdots \pi_2\pi_1$.

%%%
%%%%%%%%%%%%%%%%%%%%%%%%%%%%%%%%%%%%%%%%%%%%%%%%%%%%%%%%%%%%%%%%%%%%%%%%%%%%%%%%%%%%%%%%%%%%%%%%%%%%%%%%%%%%%%%%%%%%%%
%%%
\begin{proposition}\label{P:kk}
	Let $(G,A,\pi)$ be a linear SDS and $D=\text{Diag}\{A\}$.
	Then $(G,A,\pi)$ is invertible iff $a_{ii}\neq 0$ for all $1\leq i\leq n$ and
	\begin{align}
	(G,A,\pi)^{-1} = (G,B,\pi^{[r]}),
	\end{align}
	where $B_{\pi^{[r]}}=-D^{-1}A_{\pi^{[r]}}$ and $B_\pi=-D^{-1}A_\pi$ and $Diag\{B\}=D^{-1}$.
\end{proposition}
%%%
%%%%%%%%%%%%%%%%%%%%%%%%%%%%%%%%%%%%%%%%%%%%%%%%%%%%%%%%%%%%%%%%%%%%%%%%%%%%%%%%%%%%%%%%%%%%%%%%%%%%%%%%%%%%%%%%%%%%%%
%%%
\proof
Invertibility follows from the invertibility of the maps $F_{v_i}$ which is easily checked when considering their
matrix representation.
The matrix of the linear SDS $(G,B,\pi^{[r]})$ is
given by
\begin{eqnarray*}
	(I-(-D^{-1}A)_{\pi^{[r]}})^{-1}(D^{-1}+(-D^{-1}A)_{\pi}) & = & (I+D^{-1}A_{\pi^{[r]}})^{-1}D^{-1} (I-A_{\pi}) \\
	& = & [(I-A_{\pi})^{-1}(D+A_{\pi^{[r]}})]^{-1},
\end{eqnarray*}
which is the inverse of $(G,A,\pi)$.
\qed

\begin{remark}
	In view of $(I-A_{\pi})^{-1}(A-A_{\pi})=(I-A_{\pi})^{-1}({\rm Diag}\{A\}+A_{\pi^{[r]}})$,
	$(I-A_{\pi})^{-1}\in I(P,\mathbb{K})$ and ${\rm Diag}(A)+A_{\pi^{[r]}}$ can be considered as a function
	in the incidence algebra over the poset having the reverse order.
\end{remark}

%%%
%%%%%%%%%%%%%%%%%%%%%%%%%%%%%%%%%%%%%%%%%%%%%%%%%%%%%%%%%%%%%%%%%%%%%%%%%%%%%%%%%%%%%%%%%%%%%%%
%%%
\begin{proposition}\label{P:moeb}
	Let $P$ be a poset over $\{s_1,\dots,s_n\}$.
	For any $H\in I(P,\mathbb{K})$ such that $H_{ii}=1$ for
	$1\leq i \leq n$, there exists a linear SDS whose matrix equals $H^{-1}$.
\end{proposition}
%%%
%%%%%%%%%%%%%%%%%%%%%%%%%%%%%%%%%%%%%%%%%%%%%%%%%%%%%%%%%%%%%%%%%%%%%%%%%%%%%%%%%%%%%%%%%%%%%%%
%%%
\proof Given $P$ and $H\in I(P,\mathbb{K})$, we construct an SDS as follows:
\begin{description}
	\item[Step~$1$.] Define a graph $G$ having $\{s_1,\dots,s_n\}$ as vertices and an $G$-edge is
	placed connecting $s_i$ and $s_j$ iff they are comparable in $P$.
	\item[Step~$2$.] Let $x_i$ denote the state of the $G$-vertex $s_i$. Define
	$$
	f_{s_i}: x_i \mapsto a_{i1}x_1+a_{i2}x_2+\cdots + a_{in}x_n,
	$$
	where $a_{ii}=1$, $a_{ij}=-H_{ij}=-h(s_i,s_j)$ if $s_i < s_j$ in $P$, set $a_{ij}=0$ otherwise.
	\item[Step~$3$.] Choose the update schedule $\pi=\pi_1\cdots \pi_n$ such that $\pi^{[r]}$ is a
	linear extension of $(P,\leq)$.
\end{description}

{\it Claim.} The system map of the linear SDS $(G,A,\pi)$ equals $H^{-1}$.

To prove the claim, we observe $I-A_{\pi}=H$ and $A_{\pi^{[r]}}=0$, and compute via Theorem~\ref{2thm1}
$$
(G,A,\pi)=(I-A_{\pi})^{-1}(\text{Diag}\{A\}+A_{\pi^{[r]}})=H^{-1}I=H^{-1}.
$$
This proves the claim and whence the proposition. \qed

\begin{remark}
	The condition $H_{ii}\neq 0$ above is necessary, otherwise $H$ is not invertible as $H$ is `essentially' a triangular matrix.
	A particularly interesting function in $I(P,\mathbb{K})$ is the M\"{o}bius function $\mu$
	which is the inverse of the zeta function $\zeta \in I(P, \mathbb{K})$, $\zeta(x,y)=1$ for any
	$x\leq y$ in $P$. We can compute the M\"{o}bius function $\mu$ of any poset by
	Proposition~\ref{P:moeb} via a linear SDS.
\end{remark}

\begin{example}
	We compute the M\"{o}bius function of the poset $P$ displayed in Figure~\ref{2fig2} (a).
	\begin{figure}[!htb]
		\centering
		\includegraphics[scale=.75]{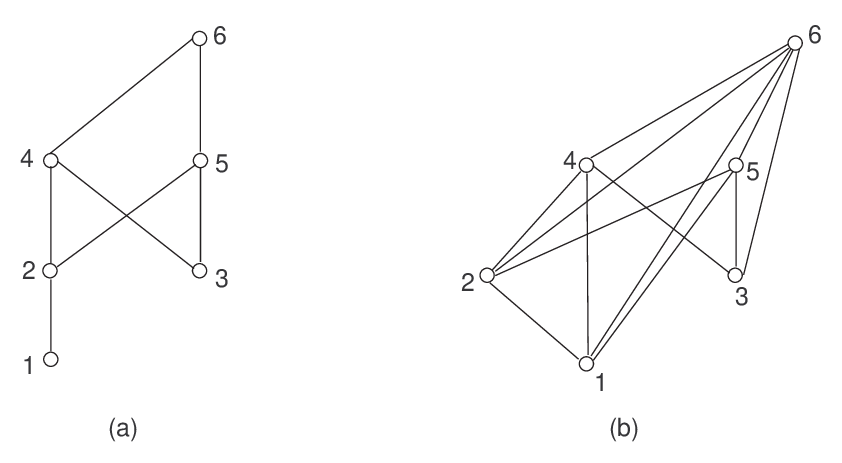}
		\caption{(a). The Hasse diagram of a poset $P$ with $6$ elements; (b). The underlying graph of the constructed SDS for $P$.}\label{2fig2}
	\end{figure}
	According to Proposition~\ref{P:moeb}, if we start with $H$ corresponding to the zeta function of $P$, then the system map of the constructed SDS equals the M\"{o}bius function $U=H^{-1}$ of $P$. The underlying graph of the SDS is illustrated in Figure~\ref{2fig2} (b).
	The local functions are encoded via the matrix $A$ and we eventually obtain $U$ as follows:

	\begin{align*}
	A=	\begin{blockarray}{ccccccc}
	1& 2 & 3 & 4& 5 & 6&\\
	\begin{block}{(cccccc)c}
	1 & -1 & 0 & -1 & -1 & -1 & 1\\
	0 & 1  & 0 & -1 & -1 & -1 & 2\\
	0 & 0  & 1 & -1 & -1 & -1 & 3\\
	0 & 0  & 0 &  1 & 0  & -1 & 4\\
	0 & 0  & 0 &  0 & 1 & -1 & 5\\
	0 & 0  & 0 &  0 & 0 & 1 & 6\\
	\end{block}
	\end{blockarray}, \qquad
	U=	\begin{blockarray}{ccccccc}
	1& 2 & 3 & 4& 5 & 6&\\
	\begin{block}{(cccccc)c}
	1 & -1 & 0 &  0 & 0 & 0 & 1\\
	0 & 1  & 0 & -1 & -1 & 1 & 2\\
	0 & 0  & 1 & -1 & -1 & 1 & 3\\
	0 & 0  & 0 &  1 & 0  & -1 & 4\\
	0 & 0  & 0 &  0 & 1 & -1 & 5\\
	0 & 0  & 0 &  0 & 0 & 1 & 6\\
	\end{block}
	\end{blockarray}
	\end{align*}
	The update schedule $\pi=654321$ is the reverse of the linear extension $123456$ of $P$.
	
	Since $U=(G,A,\pi)$ and $UI=U$, the $j$-th column of $U$ is the system state updated from the system state specified by the $j$-th column of $I$.
	For example, to compute the $4$-th column of $U$, we start with the system state $[0,0,0,1,0,0]$, and updating the vertices sequentially (following the $\pi$-order), we obtain $[0,-1,-1,1,0,0]$. In order to compute the $6$-th column of $U$, we start with the system state $[0,0,0,0,0,1]$ and update the vertices to obtain $[0,1,1,-1,-1,1]$. 
\end{example}

%%%
%%%%%%%%%%%%%%%%%%%%%%%%%%%%%%%%%%%%%%%%%%%%%%%%%%%%%%%%%%%%%%%%%%%%%%%%%%%%%%%%%%%%%%%%%%%%%%%%%%%%%%%%%%%%%%%%%%%%%%
%%%
\section{Linear SDS over words and a cut theorem of posets}\label{sec4}
%%%
%%%%%%%%%%%%%%%%%%%%%%%%%%%%%%%%%%%%%%%%%%%%%%%%%%%%%%%%%%%%%%%%%%%%%%%%%%%%%%%%%%%%%%%%%%%%%%%%%%%%%%%%%%%%%%%%%%%%%%
%%%

Suppose $\omega=w_1\cdots w_m$ is a word of $G$-vertices, where $v_i$ appears $|\omega|_{v_i}=m_i\geq 1$ times.
Clearly, we have $m=\sum_i m_i \geq n$.
Note that two adjacent vertices $v_i$ and $v_j$ may appear in $\omega$ in the pattern $v_i,\dots, v_j, \dots, v_i,\dots$.
This prevents us from directly working on a poset determined by some acyclic orientation of $G$.
In order to remedy this we proceed as follows:
%%%
%%%%%%%%%%%%%%%%%%%%%%%%%%%%%%%%%%%%%%%%%%%%%%%%%%%%%%%%%%%%%%%%%%%%%%%%%%%%%%%%%%%%%%%%%%%%%%%%%%%%%%%%%%%%%%%%%%%%%%%
%%%
\begin{definition}
	Given a matrix $T_{n\times n}=[t_{ij}]$ and a tuple $s=(m_1,\dots,m_n)$, where $\sum_k m_k=m$ and $m_k\geq 1$.
	The \emph{block-expansion} of $(T,s)$ is the matrix $^e(T,s)$ obtained from $T$ by replacing the entry $t_{ij}$ for
	$1\leq i,j\leq n$ by an $m_i\times m_j$ matrix with all entries being equal to $t_{ij}$.
\end{definition}
%%%
%%%%%%%%%%%%%%%%%%%%%%%%%%%%%%%%%%%%%%%%%%%%%%%%%%%%%%%%%%%%%%%%%%%%%%%%%%%%%%%%%%%%%%%%%%%%%%%%%%%%%%%%%%%%%%%%%%%%%%%
%%%

%%%
%%%%%%%%%%%%%%%%%%%%%%%%%%%%%%%%%%%%%%%%%%%%%%%%%%%%%%%%%%%%%%%%%%%%%%%%%%%%%%%%%%%%%%%%%%%%%%%%%%%%%%%%%%%%%%%%%%%%%%%
%%%
\begin{definition}
	Given a matrix $T_{m\times m}=[t_{ij}]$ and a tuple $s=(m_1,\dots,m_n)$, where $\sum_k m_k=m$ and $m_k\geq 1$.
	The \emph{block-compression} of $(T,s)$ is the matrix $^c(T,s)$ obtained from $T$ by replacing the
	block in $T$ bounded by the $(1+\sum_{k<r}m_{k})$-th row, the $(\sum_{k\leq r }m_k)$-th row, the
	$(1+\sum_{k<l}m_k)$-th column and the $(\sum_{k\leq l}m_k)$-th column, by the sum over all the entries in
	the block for all $1\leq r,l\leq n$, where $m_0$ is assumed to be $0$.
\end{definition}
%%%
%%%%%%%%%%%%%%%%%%%%%%%%%%%%%%%%%%%%%%%%%%%%%%%%%%%%%%%%%%%%%%%%%%%%%%%%%%%%%%%%%%%%%%%%%%%%%%%%%%%%%%%%%%%%%%%%%%%%%%%
%%%

\begin{example}
	Let $s=(2,2)$.
	Then,
	\begin{align*}
	T=\left(\begin{array}{cc}
	1&2\\
	3&4
	\end{array}\right)\qquad
	&\Longrightarrow \quad
	^e(T,s)=\left(\begin{array}{cc:cc}
	1&1&2&2\\
	1&1&2&2\\ \hdashline
	3&3&4&4\\
	3&3&4&4
	\end{array}\right),\\
	T=\left(\begin{array}{cc:cc}
	1&1&3&2\\
	3&3&2&2\\ \hdashline
	5&3&4&4\\
	3&4&5&4
	\end{array}\right)\qquad
	&\Longrightarrow \quad
	^c(T, s)=\left(\begin{array}{cc}
	8&9\\
	15&17
	\end{array}\right).
	\end{align*}
\end{example}

%%%
%%%%%%%%%%%%%%%%%%%%%%%%%%%%%%%%%%%%%%%%%%%%%%%%%%%%%%%%%%%%%%%%%%%%%%%%%%%%%%%%%%%%%%%%%%%%%%%%%%%%%%%%%%%%%%
%%%
\begin{theorem}\label{4thm1}
	Let $(G,A,\omega)$ be a linear SDS over a word $\omega=w_1\cdots w_m$ where $|\omega|_{v_i}=m_i \geq 1$.
	Let a tuple $s=(m_1,\dots, m_n)$ and
	$B=~^e(A-I,s)$ be a matrix with both rows and columns indexed by $u_1,\ldots, u_m$,
	and let $\bar{\omega}=\bar{w}_1 \cdots
	\bar{w}_m$ be a sequence on $u_i$'s, obtained by substituting the $k$-th appearance of $v_i$ in $\omega$ by
	$u_{\sum_{j=1}^{i-1} m_j+k}$. Then,
	\begin{align}\label{4t1}
	(G,A,\omega)=I+~^c((I-B_{\bar{\omega}})^{-1},s)(A-I).
	\end{align}
\end{theorem}
%%%
%%%%%%%%%%%%%%%%%%%%%%%%%%%%%%%%%%%%%%%%%%%%%%%%%%%%%%%%%%%%%%%%%%%%%%%%%%%%%%%%%%%%%%%%%%%%%%%%%%%%%%%%%%%%%%
%%%
\proof First we have
\begin{align*}
(G,A,\omega)=F_{w_m}F_{w_{m-1}}\cdots F_{w_1}&=(I+\bar{F}_{w_m})(I+\bar{F}_{w_{m-1}})\cdots (I+\bar{F}_{w_1})\\
&=I+\sum_{l=1}^{m}\sum_{m\geq k_l>\cdots>k_1\geq 1}\bar{F}_{w_{k_l}}\cdots\bar{F}_{w_{k_1}}.
\end{align*}
Let $s=s_1\cdots s_k$ be a sequence. If $s'=s_{i1}\cdots s_{ij}$ is a subsequence of $s$,
we write $s'\lhd s$.

%%%
%%%%%%%%%%%%%%%%%%%%%%%%%%%%%%%%%%%%%%%%%%%%%%%%%%%%%%%%%%%%%%%%%%%%%%%%%%%%%%%%%%%%%%%%%%%%%%%%%%%%%%%%%%%%%%
%%%
{\it Claim~$1$.} For $l \geq 1$, $n\geq i,j \geq 1$,
\begin{align}\label{eq31}
\Big[ \sum_{m\geq k_l>\cdots>k_1\geq 1}\bar{F}_{w_{k_l}}\cdots\bar{F}_{w_{k_1}} \Big]_{ij}=
\sum_{v_{d_{1}}\dots v_{d_{l-1}}v_i \lhd \omega} [\bar{F}_{v_i}]_{i d_{l-1}}
[\bar{F}_{v_{d_{l-1}}}]_{d_{l-1} d_{l-2}} \cdots [\bar{F}_{v_{d_1}}]_{d_1 j},
\end{align}
%%%
%%%%%%%%%%%%%%%%%%%%%%%%%%%%%%%%%%%%%%%%%%%%%%%%%%%%%%%%%%%%%%%%%%%%%%%%%%%%%%%%%%%%%%%%%%%%%%%%%%%%%%%%%%%%%%
%%%
where the summation on the RHS is over all length-$l$ subsequences ending with $v_i$ in $\omega$ (note here and
in the following, a subsequence may appear multiple times due to the multiplicity of vertices in $\omega$), and
interpreted as $\sum_{v_i \in  \omega} [\bar{F}_{v_i}]_{ij}$, i.e.~summation over all occurrences of $v_i$,
for $l=1$.
The LHS of eq.~\eqref{eq31} is formally equal to a summation over all subsequences of length $l$.
It reduces to the RHS by the argument given by Claim~$1$ in Theorem~\ref{2thm1}, the difference being that
some vertices may appear multiple times in a subsequence.

%%%
%%%%%%%%%%%%%%%%%%%%%%%%%%%%%%%%%%%%%%%%%%%%%%%%%%%%%%%%%%%%%%%%%%%%%%%%%%%%%%%%%%%%%%%%%%%%%%%%%%%%%%%%%%%%%%
%%%
{\it Claim~$2$.} For $l \geq 1$, $n\geq i,k \geq 1$,
\begin{align}\label{eq32}
\sum_{v_k v_{j_{1}}\dots v_{j_{l-1}}v_i \lhd \omega}
[\bar{F}_{v_i}]_{i j_{l-1}} [\bar{F}_{v_{j_{l-1}}}]_{j_{l-1} j_{l-2}} \cdots [\bar{F}_{v_{j_1}}]_{j_1 k}
=[~^c(B_{\bar{\omega}}^{l},s)]_{ik}.
\end{align}
%%%
%%%%%%%%%%%%%%%%%%%%%%%%%%%%%%%%%%%%%%%%%%%%%%%%%%%%%%%%%%%%%%%%%%%%%%%%%%%%%%%%%%%%%%%%%%%%%%%%%%%%%%%%%%%%%%
%%%
Note that the summation of the LHS of eq.~\eqref{eq32} is over all length-$(l+1)$ subsequences starting with
$v_k$ and ending with $v_i$ in $\omega$. Fixing an occurrence of $v_i$ and $v_k$ in $\omega$ (e.g., the leftmost pair
of $v_i$ and $v_k$ in the following illustration),
\begin{align*}
\vcenter{\xymatrix@C=0pc@R=1pc{
		v\text{-sequence}~\omega: & &  w_1 & w_{2}& \cdots &v_k\ar@{->}[d]&\cdots & v_i\ar@{->}[d]&\cdots & v_k &\cdots &
		v_i & \cdots & w_{m-1} & w_m \\
		u\text{-sequence}~\bar{\omega}: && \bar{w}_1 & \bar{w}_{2}& \cdots &u_q &\cdots & u_p&\cdots &
		u_{\bar{q}} &\cdots & u_{\bar{p}} & \cdots & \bar{w}_{m-1} & \bar{w}_m
	}}
	\end{align*}
	we compute the summation over all length-$(l+1)$ subsequences
	starting with this $v_k$ and ending with this $v_i$:
	fixing a subsequence $v_k v_{j_1} \dots v_{j_{l-1}} v_i$ in $\omega$, this
	subsequence corresponds to a subsequence $u_q u_{j_1'}\dots u_{j_{l-1}'} u_p$ in $\bar{\omega}$. Let $R_{u_i}$ be the matrix
	induced by the row indexed by $u_i$ in $B$, in analogy to the construction of $F_{v_i}$ induced by the $i$-th
	row of $A$. Let $R_{u_i}=I+\bar{R}_{u_i}$. Then, by construction,	
	\begin{align*}
	[\bar{F}_{v_i}]_{i j_{l-1}} [\bar{F}_{v_{j_{l-1}}}]_{j_{l-1} j_{l-2}} \cdots [\bar{F}_{v_{j_1}}]_{j_1 k}
	&=[\bar{R}_{u_p}]_{p j'_{l-1}}[\bar{R}_{u_{j'_{l-1}}}]_{j'_{l-1} j'_{l-2}} \cdots [\bar{R}_{u_{j'_1}}]_{j'_1 q}.
	\end{align*}
	Thus, the summation over all length-$(l+1)$ subsequences starting with this $v_k$ and ending with this $v_i$ is
	tantamount to summing over all length-$(l+1)$ subsequences starting with $u_q$ and ending with $u_p$ in $\bar{\omega}$,
	that is,
	\begin{align}\label{4eq16}
	&\sum_{\substack{v_kv_{j_1}\dots v_{j_{l-1}v_i \lhd \omega}\\ v_k\rightarrow u_q,v_i\rightarrow u_p}}
	[\bar{F}_{v_i}]_{i j_{l-1}} [\bar{F}_{v_{j_{l-1}}}]_{j_{l-1} j_{l-2}} \cdots [\bar{F}_{v_{j_1}}]_{j_1 k}\nonumber\\
	=&\sum_{u_q u_{j'_1}\dots u_{j'_{l-1}} u_p \lhd \bar{\omega}} [\bar{R}_{u_p}]_{p j'_{l-1}}[\bar{R}_{u_{j'_{l-1}}}]_{j'_{l-1} j'_{l-2}}
	\cdots [\bar{R}_{u_{j'_1}}]_{j'_1 q}
	= [B_{\bar{\omega}}^{l}]_{pq},
	\end{align}
	where the last equation follows from Lemma~\ref{2lem1}, in view of the fact
	that $\bar{\omega}$ induces a partial order on $u$-elements such that $u_p< u_{j_{l-1}'}<u_{j_{l-2}'}<\cdots < u_{j'_1} <u_q$.

	To prove Claim~$2$, we sum over all pairs of $v_i,v_k$ such that $v_i$ appears after $v_k$ in $\omega$ on the
	LHS of eq.~\eqref{4eq16},
	which implies summing over all $p,q$ such that $u_p$ is a substitution of $v_i$, $u_q$ is a substitution of $v_k$,
	and $u_p$ appears after $u_q$ in $\bar{\omega}$ on the RHS of eq.~\eqref{4eq16}.
	Note if $p,q$ are such a pair for the latter, then $[B_{\bar{\omega}}^{l}]_{qp}=0$ as there is no chain (of length $l+1$)
	directed from
	$u_q$ to $u_p$ in the partial order induced by $\bar{\omega}$, thus the latter equals to
	summing over all $p,q$ such that $u_p$ is a substitution of $v_i$ and $u_q$ is a substitution of $v_k$ in $\bar{\omega}$.
	By construction, all of the $m_i$ $u$-elements corresponding to $v_i$ and $m_k$ $u$-elements corresponding to $v_k$, respectively, are consecutively
	arranged in $B$ (thus $B_{\bar{\omega}}^l$),
	whence the summation over all $p,q$ means to sum over an $m_i\times m_k$
	block in $B_{\bar{\omega}}^l$. Accordingly, we arrive at
	$$
	\sum_{\substack{1\leq p,\; q\leq m\\
			v_p, \; \text{\rm a subst. of $v_i$}\\
			v_q, \; \text{\rm a subst. of $v_k$}}}
	[B_{\bar{\omega}}^{l}]_{pq}=[~^c(B_{\bar{\omega}}^{l},s)]_{ik},
	$$
	and Claim~$2$ is proved.

	Applying Claim~$1$ and Claim~$2$, we compute
	\begin{align*}
	&\quad \Big[ \sum_{m\geq k_l>\cdots>k_1\geq 1}\bar{F}_{w_{k_l}}\cdots\bar{F}_{w_{k_1}} \Big]_{ij}
	=\sum_{v_{d_{1}}\dots v_{d_{l-1}}v_i \lhd \omega} [\bar{F}_{v_i}]_{i d_{l-1}}
	[\bar{F}_{v_{d_{l-1}}}]_{d_{l-1} d_{l-2}} \cdots [\bar{F}_{v_{d_1}}]_{d_1 j}\\
	& =\sum_{d_1=1}^n \Bigg( \sum_{v_{d_1}\dots v_{d_{l-1}}v_i \lhd \omega}[\bar{F}_{v_i}]_{i d_{l-1}}
	[\bar{F}_{v_{d_{l-1}}}]_{d_{l-1} d_{l-2}} \cdots [\bar{F}_{v_{d_{2}}}]_{d_{2} d_{1}} \Bigg) [\bar{F}_{v_{d_1}}]_{d_1 j}
	=\sum_{d_1=1}^n [~^c(B_{\bar{\omega}}^{l-1}, s)]_{i d_1} [A-I]_{d_1 j}\\
	& =[~^c(B_{\bar{\omega}}^{l-1}, s)(A-I)]_{ij}.
	\end{align*}
	Therefore,
	\begin{align*}
	(G,A,\omega)=I+~^c(B_{\bar{\omega}}^{0}, s)(A-I)+~^c(B_{\bar{\omega}}^{1}, s)(A-I)+\cdots +~^c(B_{\bar{\omega}}^{m-1}, s)(A-I).
	\end{align*}
	In view of $^c(A,s)+~^c\/(B,s)=~^c(A+B, s)$, we observe
	$$
	^c(B_{\bar{\omega}}^{0}, s)+~^c(B_{\bar{\omega}}^{1}, s)+\cdots +~^c(B_{\bar{\omega}}^{m-1}, s)
	=~^c(B_{\bar{\omega}}^{0}+B_{\bar{\omega}}^{1}+\cdots+B_{\bar{\omega}}^{m-1},s).
	$$
	Since $B_{\bar{\omega}}^{0}+B_{\bar{\omega}}^{1}+\cdots+B_{\bar{\omega}}^{m-1}=(I-B_{\bar{\omega}})^{-1}$ we obtain
	$$
	(G,A,\omega)=I+~^c((I-B_{\bar{\omega}})^{-1},s)(A-I),
	$$
	completing the proof of the theorem.\qed

	\begin{example}
		Consider the linear SDS $(G,f,\omega)$ where $G,\ f$ are as defined in Example~\ref{eg31} with update schedule $\omega=013120321$. Following the construction,
		$\bar{\omega}=u_1u_3u_8u_4u_6u_2u_9u_7u_5$ and
		$s=(2,3,2,2)$,
		\begin{align*}
		A=\left(\begin{array}{cccc}
		1&1&0&1\\
		1&1&1&0\\
		0&1&1&1\\
		1&0&1&1
		\end{array}\right),\quad	
		B=~^e(A-I,s)=\left(\begin{array}{cc:ccc:cc:cc}
		0& 0 & 1 &1 &1 &0&0&1&1\\
		0& 0 & 1 &1 &1 &0&0&1&1\\ \hdashline
		1& 1 & 0 &0 &0 &1&1&0&0\\
		1& 1 & 0 &0 &0 &1&1&0&0\\
		1& 1 & 0 &0 &0 &1&1&0&0\\ \hdashline
		0& 0 & 1& 1 & 1 &0 &0 &1&1\\
		0& 0 & 1& 1 & 1 &0 &0 &1&1\\ \hdashline
		1& 1 & 0&0 &0 & 1& 1 & 0 &0\\
		1& 1 & 0&0 &0 & 1& 1 & 0 &0
		\end{array}\right),\quad
		\end{align*}
		\begin{align*}
		B_{\bar{\omega}}=\left(\begin{array}{cc:ccc:cc:cc}
		0& 0 & 0 &0 &0 &0&0&0&0\\
		0& 0 & 1 &1 &0 &0&0&1&0\\ \hdashline
		1& 0 & 0 &0 &0 &0&0&0&0\\
		1& 0 & 0 &0 &0 &0&0&0&0\\
		1& 1 & 0 &0 &0 &1&1&0&0\\ \hdashline
		0& 0 & 1& 1 & 0 &0 &0 &1&0\\
		0& 0 & 1& 1 & 0 &0 &0 &1&1\\ \hdashline
		1& 0 & 0&0 &0 & 0& 0 & 0 &0\\
		1& 1 & 0&0 &0 & 1& 0 & 0 &0
		\end{array}\right),\quad
		(I-B_{\bar{\omega}})^{-1}=
		\left(\begin{array}{cc:ccc:cc:cc}
		1& 0 & 0 &0 &0 &0&0&0&0\\
		1& 1 & 1 &1 &0 &0&0&1&0\\ \hdashline
		1& 0 & 1 &0 &0 &0&0&0&0\\
		1& 0 & 0 &1 &0 &0&0&0&0\\
		1& 0 & 1 &1 &1 &0&1&1&1\\ \hdashline
		1& 0 & 1& 1 & 0 &1 &0 &1&0\\
		0& 1 & 1& 1 & 0 &1 &1 &1&1\\ \hdashline
		1& 0 & 0&0 &0 & 0& 0 & 1 &0\\
		1& 1 & 0&0 &0 & 1& 0 & 0 &1
		\end{array}\right).
		\end{align*}
		Then,
		\begin{align*}
		(G,A,\omega)=I+~^c((I-B_{\bar{\omega}})^{-1},s)(A-I)
		=I+\left(\begin{array}{cccc}
		1&0&0&1\\
		1&1&1&0\\
		0&0&1&1\\
		1&0&1&0
		\end{array}\right)(A-I).
		\end{align*}
	\end{example}

	Suppose the update schedule $\omega$ is the concatenation of two words $\omega_1$ and $\omega_2$, i.e., $\omega=\omega_1 \omega_2$,
	such that each $G$-vertex appears at least once in both. Clearly, we have $(G,A,\omega)=(G,A,\omega_2)\circ
	(G,A,\omega_1)$. Applying Theorem~\ref{4thm1} leads to the relation
	\begin{multline}\label{4eq15}
	I+~^c((I-B_{\bar{\omega}})^{-1},s)(A-I)\\
	=\Big(I+~^c\big((I-B_{\bar{\omega}_2})^{-1},s_2\big)(A-I)\Big)\Big(I+~^c\big((I-B_{\bar{\omega}_1})^{-1},s_1\big)(A-I)\Big),
	\end{multline}
	where $s$, $s_1$ and $s_2$ encode the multiplicities of the vertices appearing in $\omega$, $\omega_1$ and
	$\omega_2$, respectively.
	
	Note a matrix of the form $I-T_{\pi}$ (recall the construction of $A_{\pi}$ in Section~\ref{sec3}) can always be interpreted as an element in $I(P,\mathbb{K})$ for some $P$. Thus, eq.~\eqref{4eq15}
	implies a relation among three posets (induced by $\omega$, $\omega_1$ and $\omega_2$, respectively).
	In particular, if $I-B_{\bar{\omega}}$, $I-B_{\bar{\omega_1}}$ and $I-B_{\bar{\omega_2}}$ are the zeta functions of the
	posets, eq.~\eqref{4eq15} expresses a relation of the corresponding M\"{o}bius
	functions.
	
	Our next objective is to study this relation in more detail, i.e.,~explicitly derive these three posets and how eq.~\eqref{4eq15} implies a relation among their M\"{o}bius functions.
	As a result, we obtain a cut theorem of posets w.r.t. certain chain decompositions.

	Given a graph $G$ and a word $\omega=w_1\cdots w_m$, where $|\omega|_{v_i}=m_i \geq 1$, we introduce the graph
	$G_{\omega}$ with $V(G_{\omega})=\{u_1,\ldots, u_m\}$ via the following data:
	\begin{itemize}
		\item a mapping $\phi: V(G_{\omega}) \rightarrow V(G)$, where $\phi^{-1}(v_i)=\{u_{1+\sum_{j=0}^{i-1}m_{j}},\ldots, u_{m_i+\sum_{j=0}^{i-1}m_{j}}\}$ and $m_0=0$,
		\item $u_i$ and $u_j$ are adjacent in $G_{\omega}$ if $\phi(u_i)$ and $\phi(u_j)$ are adjacent in $G$, or if
		$\phi(u_i)=\phi(u_j)$.
	\end{itemize}
	Let $\bar{\omega}$ be defined as in Theorem~\ref{4thm1}. Then, $\bar{\omega}$ determines an acyclic orientation of
	$G_{\omega}$, i.e.~a poset $P$, and the matrix $I-B_{\bar{\omega}}\in I(P,\mathbb{K})$. Furthermore, all $u$-elements
	that map into a fixed $v$ form a $G_{\omega}$-clique, representing a particular $P$-chain, which we call a
	\emph{clique-chain}. 
	The construction here is related to heaps~\cite{stemb, viennot}, if we consider the acyclic orientation of $G_{\omega}$ (induced by an update schedule) and the map $\phi:V(G_{\omega})\rightarrow V(G)$ together. It is a very particular heap though, as in a general heap, the involved graph and poset are usually independent instead of having a close relation. 
	
	If $\omega=\omega_1\omega_2$, then
	the above construction gives rise to three graphs, $G_\omega$, $G_{\omega_1}$ and $G_{\omega_2}$, and $G_{\omega_1}$ and $G_{\omega_2}$ can naturally be considered as
	subgraphs of $G_\omega$, such that any $P$ clique-chain $C=(\xi_1<\dots<\xi_s)$ is cut into the two
	clique chains $(\xi_1<\dots<\xi_{h(C)})$ and $(\xi_{h(C)+1}<\dots<\xi_s)$.
	The latter are elements in the posets $P_2(\omega_2)$ and $P_1(\omega_1)$, respectively. By construction, the acyclic orientations on
	$G_{\omega_i}$ that are determined by $\bar{\omega}_i$ are compatible with $Acyc(G_{\omega},\bar{\omega})$. Thus,
	$P_i$ is a subposet of $P$, and all elements in $P_2$ are smaller than $P_1$-elements in $P$, if they are comparable.
	
	\begin{example}
		Consider the graph $G= C_4$ and the update schedule $\omega=013120321$. Following the
		construction introduced above,
		the graph $G_{\omega}$ is shown in Figure~\ref{4fig1} and $\bar{\omega}=u_1u_3u_8u_4u_6u_2u_9u_7u_5$.
		Let $\omega_1=01312$ and $\omega_2=0321$. Accordingly, $\bar{\omega}_1=u_1u_3u_8u_4u_6$ and
		$\bar{\omega}_2=u_2u_9u_7u_5$. The graph $G_{\omega_1}$ is the subgraph induced by the vertices $u_1,u_3,u_8,u_4,u_6$
		while $G_{\omega_2}$ is the subgraph induced by the rest of vertices.
		\begin{figure}[!htb]
			\centering
			\includegraphics[scale=.75]{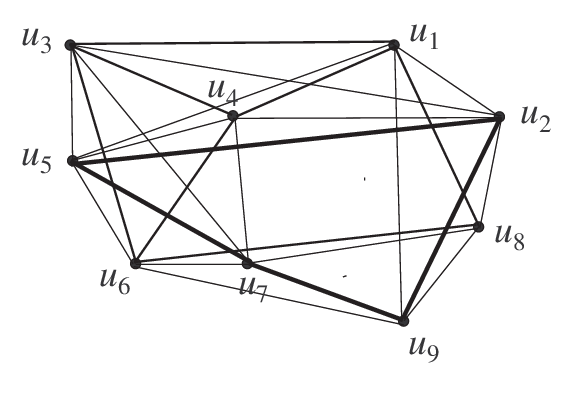}
			\caption{The graph $G_{\omega}$ constructed from $C_4$ and $\omega=013120321$.}\label{4fig1}
		\end{figure}
	\end{example}

	%%%
	%%%%%%%%%%%%%%%%%%%%%%%%%%%%%%%%%%%%%%%%%%%%%%%%%%%%%%%%%%%%%%%%%%%%%%%%%%%%%%%%%%%%%%%%%%%%%%%%%%%%%
	%%%
	\begin{definition}
		Let $(P,\leq)$ be a poset. A \emph{homogeneous chain decomposition} $\mathcal{C}$ of $P$ is a collection of mutually disjoint
		chains $\{C_1,\ldots, C_n\}$ such that $\bigcup_i C_i=P$, and if $x\in C_i$ and $y\in C_j$ are
		comparable, then all elements in $C_i$ and $C_j$ are comparable.
	\end{definition}
	%%%
	%%%%%%%%%%%%%%%%%%%%%%%%%%%%%%%%%%%%%%%%%%%%%%%%%%%%%%%%%%%%%%%%%%%%%%%%%%%%%%%%%%%%%%%%%%%%%%%%%%%%%
	%%%
	
	%%%
	%%%%%%%%%%%%%%%%%%%%%%%%%%%%%%%%%%%%%%%%%%%%%%%%%%%%%%%%%%%%%%%%%%%%%%%%%%%%%%%%%%%%%%%%%%%%%%%%%%%%%
	%%%
	\begin{definition}
		Suppose $\mathcal{C}=\{C_1,\ldots, C_n\}$ is a homogeneous chain decomposition of $P$.
		A \emph{$\mathcal{C}$-cut} of $P$ is a cut of each chain $C_i=(\xi_1<\dots<\xi_s)$ into
		$C_i^{low}=(\xi_1<\dots<\xi_{h(i)})$ and $C_i^{up}=(\xi_{h(i)+1}<\dots<\xi_{s})$
		such that elements in $C_i^{low}$ are smaller than elements in $C_j^{up}$ if comparable.
	\end{definition}
	%%%
	%%%%%%%%%%%%%%%%%%%%%%%%%%%%%%%%%%%%%%%%%%%%%%%%%%%%%%%%%%%%%%%%%%%%%%%%%%%%%%%%%%%%%%%%%%%%%%%%%%%%%
	%%%
	Any $\mathcal{C}$-cut of $P$ induces the subposets $P^{low}$ and $P^{up}$, in which
	$\mathcal{C}^{low}=\{C_1^{low},\ldots, C_n^{low}\}$ and $\mathcal{C}^{up}=\{C_1^{up},\ldots, C_n^{up}\}$
	are homogeneous chain decompositions of $P^{low}$ and $P^{up}$, respectively.
	See Figure~\ref{3fig2} for an example, where elements on the same vertical form a chain.
	
	\begin{figure}[!htb]
		\centering
		\includegraphics[scale=.6]{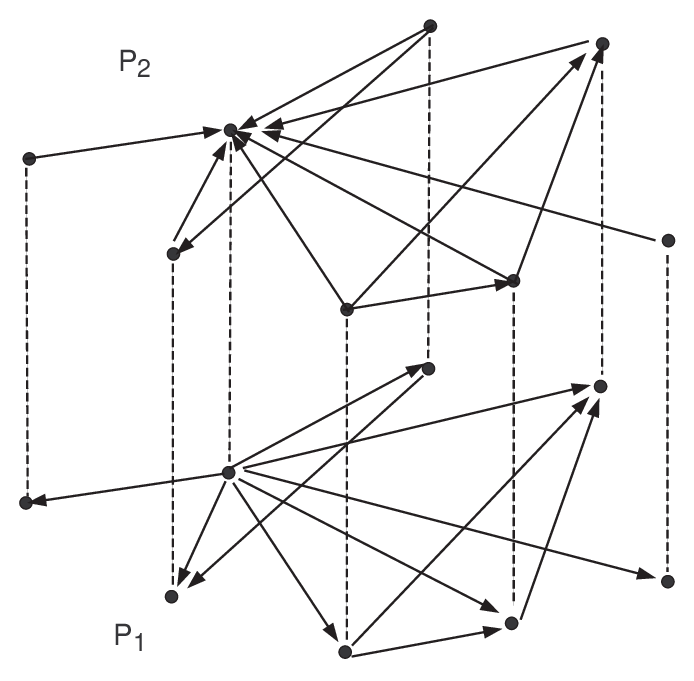}
		\caption{Making a cut on each chain (vertical dashed line) of $P$ leads to the upper sub-poset $P_2$ and the lower sub-poset $P_1$, where all elements in $P_1$ are assumed to be smaller than those in $P_2$ if comparable although the corresponding arrows are not shown here.}\label{3fig2}
	\end{figure}

	Let $G_{\mathcal{C}}$ be the \emph{$\mathcal{C}$-graph} of $P$, having the chains $C_i$ as vertices and
	in which $C_i$ and $C_j$ are adjacent iff their respective elements are comparable. Let $J=I+J_0$, where $J_0$ is the adjacency matrix of the $\mathcal{C}$-graph indexed by $C_1,\ldots, C_n$,
	i.e., $[J_0]_{ij}=1$ if $C_i$ and $C_j$ ($i\neq j$) are adjacent in $G_{\mathcal{C}}$, $[J_0]_{ij}=0$,
	otherwise.
	
	In the homogeneous chain decomposition of $P$, suppose 
	$C_i=\{s_{i1},\ldots ,s_{im_i}\}$ and $s_{ij}<s_{i(j+1)}$.
	Let $W=w_1\cdots w_N$ be any linear extension of $P$ and let $\omega$
	be a word on the vertex set of $G_{\mathcal{C}}$ which is obtained from $W^{[r]}$ (i.e., the reverse of $W$) by replacing $w_i$ with $C_j$ if $w_i$ belongs to $C_j$.
	Let $U$ be the matrix of the M\"{o}bius function of $P$ with the rows and columns indexed by
	$s_{11},\ldots, s_{1m_1},s_{21},\ldots, s_{2m_2},\ldots, s_{n1},\ldots, s_{nm_n}$.
	Let $s=(m_1,\ldots, m_n)$. Then, we have the following lemma:

	\begin{lemma}\label{4lem2}
		The matrix form of the linear SDS $(G_{\mathcal{C}},I-J,\omega)$ is  $I-~^c(U,s)J$.
	\end{lemma}
	
	\begin{proof}
		Theorem~\ref{4thm1} allows us to compute $(G_{\mathcal{C}},I-J,\omega)$ as follows:
		let $B=~^e(-J,s)$ with rows and columns indexed by $s_{11}, \dots, s_{1m_1}, s_{21}, \dots$ (corresponding to those $u$'s in Theorem~\ref{4thm1}).
		The sequence $\bar{\omega}$ induced by $\omega$ is then $W^{[r]}$ (i.e., the reverse of $W$).
		By construction
		$$
		[B_{\bar{\omega}}]_{s_{i_1j_1}s_{i_2j_2}}=
		\begin{cases}
		-1, & \mbox{if $s_{i_1j_1}$ appears before $s_{i_2j_2}$ in $W$, and $C_{i_1}$ and $C_{i_2}$ are comparable}; \\
		\  \  0, & \text{otherwise.}
		\end{cases}
		$$
		In other words, 
		$$
		[I-B_{\bar{\omega}}]_{s_{i_1j_1}s_{i_2j_2}}=
		\begin{cases}
		1, & \mbox{if $s_{i_1j_1}<s_{i_2j_2}$ in $P$}; \\
		0, & \text{otherwise,}
		\end{cases}
		$$
		so that $I-B_{\bar{\omega}}$ is the matrix of the zeta function of $P$ and
		$(I-B_{\bar{\omega}})^{-1}$ is the matrix of the M\"{o}bius function, $U$.
		According to Theorem~\ref{4thm1} we have
		$$
		(G_{\mathcal{C}},I-J,\omega)=I+~^c((I-B_{\bar{\omega}})^{-1},s)(-J+I-I)=I+~^c(U,s)(-J),
		$$
		completing the proof.
	\end{proof}

	Suppose a $\mathcal{C}$-cut leads to $P^{low}$ and $P^{up}$, such that each $C_i$ is split into
	\begin{eqnarray*}
		C_i^{low}  =   \{s_{i1},\ldots, s_{im_i'}\}, \quad 
		C_i^{up}   =  \{s_{im_i'+1},\ldots, s_{im_i}\}.
	\end{eqnarray*}
	Let $U^{low}$ be the matrix of the M\"{o}bius function of $P^{low}$ with the rows and columns being indexed
	by
	$$
	s_{11}, \ldots, s_{1m_1'}, s_{21}, \ldots, s_{2m_2'}, \ldots, s_{n1}, \ldots, s_{nm_n'}
	$$
	and
	$U^{up}$ be the matrix of the M\"{o}bius function of $P^{up}$ with the rows and columns indexed
	by
	$$
	s_{1(m_1'+1)}, \ldots, s_{1m_1}, s_{2(m_2'+1)}, \ldots, s_{2m_2}, \ldots, s_{n(m_n'+1)}, \ldots, s_{nm_n}.
	$$
	
	\begin{lemma}\label{4lem3}
		Given $P$ and a $\mathcal{C}$-cut, specifying $P^{low},P^{up}$, there exists a linear extension $W$
		of $P$ such that $W$ is a concatenation of $W_1$ and $W_2$, i.e.,  $W=W_1W_2$, where $W_1$ and $W_2$
		are linear extensions of $P^{low}$ and $P^{up}$, respectively. 
	\end{lemma}
	
	\begin{proof}
		By assumption, all $P^{low}$-elements are smaller than $P^{up}$-elements if they are comparable
		in $P$. Then, there certainly exists a linear extension of $P$ where the elements in $P^{low}$ come before the elements in $P^{up}$. It is clear that any such linear extension can be written as a concatenation of a linear extension of $P^{low}$ and a linear extension of $P^{up}$, 
		whence the lemma.
	\end{proof}

	Accordingly, starting with an arbitrary poset $P$, we construct the graph $G_C$, whose vertices
	are $P$-chains. By Lemma~\ref{4lem2}, any linear SDS over $G_C$ can be expressed as a matrix and by
	Lemma~\ref{4lem3}, given $P$ and a $\mathcal{C}$-cut, $P$ has a linear extension that can be decomposed 
	into a concatenation of linear extensions of $P^{low}$ and $P^{up}$, respectively. 
	
	As a result, we obtain the following cut theorem:

	%%%
	%%%%%%%%%%%%%%%%%%%%%%%%%%%%%%%%%%%%%%%%%%%%%%%%%%%%%%%%%%%%%%%%%%%%%%%%%%%%%%%%%%%%%%%%%%%%%%%%%%%%%%%%%%%%
	%%%
	\begin{theorem}[Cut theorem]\label{T:rel}
		Let $(P,\leq)$ be a poset having the homogeneous chain decomposition $\mathcal{C}=\{C_1,\dots, C_n\}$ and
		a $\mathcal{C}$-cut inducing posets $P^{low}$ and $P^{up}$. We further assume the $\mathcal{C}$-cut splits each
		chain $C_i=\{s_{i1},\ldots ,s_{im_i}\}$ into $	C_i^{low}  =   \{s_{i1},\ldots, s_{im_i'}\}$ and 
		$C_i^{up}   =  \{s_{im_i'+1},\ldots, s_{im_i}\}$.
		Let $s=(m_1,\ldots, m_n)$, $s^{low}=(m_1',\ldots, m_n')$ and $s^{up}=(m_1'',\ldots, m_n'')$,
		where $m_i'+m_i''=m_i$. Then, we have
		\begin{align}
		^c(U,s)J=~^c(U^{low},s^{low})J~+~^c(U^{up},s^{up})J~-~^c(U^{low},s^{low})J~^c(U^{up},s^{up})J,
		\end{align}
		where $J=I+J_0$ and $J_0$ is the adjacency matrix of the $\mathcal{C}$-graph $G_{\mathcal{C}}$. In particular, if $J$ is invertible, we have
		\begin{align}
		^c(U,s)=~^c(U^{low},s^{low})~+~^c(U^{up},s^{up})~-~^c(U^{low},s^{low})J~^c(U^{up},s^{up}).
		\end{align}
	\end{theorem}
	%%%
	%%%%%%%%%%%%%%%%%%%%%%%%%%%%%%%%%%%%%%%%%%%%%%%%%%%%%%%%%%%%%%%%%%%%%%%%%%%%%%%%%%%%%%%%%%%%%%%%%%%%%%%%%%%%
	%%%
	
	\begin{proof}
		According to Lemma~\ref{4lem3}, we can pick a linear extension $W$ of $P$, such that $W=W_1W_2$ where $W_1$ and $W_2$
		are linear extensions of $P^{low}$ and $P^{up}$, respectively.
		Then, the corresponding word $\omega$
		(associated with $W^{[r]}$) can be split into two words $\omega_1$ and $\omega_2$ such that
		$\omega_i$ is induced by the reverse order of $W_i$.
		Note that the ``$G_{\mathcal{C}}$-graphs" induced by the homogeneous chain decompositions (induced by the cut) of
		$P^{low}$ and $P^{up}$ are exactly the same as $G_{\mathcal{C}}$. Hence, by Lemma~\ref{4lem2},
		we have
		\begin{align*}
		(G_{\mathcal{C}},I-J,\omega_1) =I+~^c(U^{low},s^{low})(-J)\quad \text{and}\quad
		(G_{\mathcal{C}},I-J,\omega_2) =I+~^c(U^{up},s^{up})(-J).
		\end{align*}
		
		The $\omega$-update can be viewed as updating via $\omega_2$ first and then updating via $\omega_1$ and thus
		is tantamount to composing the linear SDS:
		$$
		(G,A,\omega)=(G,A,\omega_1)\circ (G,A,\omega_2).
		$$
		We consequently arrive at the identity
		$$
		I+~^c(U,s)(-J)=(I+~^c(U^{low},s^{low})(-J))(I+~^c(U^{up},s^{up})(-J)),
		$$
		which is equivalent to
		$$
		^c(U,s)J=~^c(U^{low},s^{low})J~+~^c(U^{up},s^{up})J~-~^c(U^{low},s^{low})J~^c(U^{up},s^{up})J.
		$$
	\end{proof}

	\section{Discussion}\label{sec5}
	
	In this paper, we studied linear SDS and their connection with incidence algebras of posets.
	On the one hand, we employed the incidence algebra of posets, in order to derive an explicit
	closed formula for any linear sequential dynamical system as a synchronous linear dynamical
	system. We furthermore proved constructively that for any synchronous linear system, there exists
	a linear SDS equivalent.
	On the other hand, SDS provided insight into posets: in Proposition~\ref{P:moeb} we computed the M\"{o}bius
	function of an arbitrary poset via an SDS. In addition, we presented a cut theorem for the M\"{o}bius
	function of posets w.r.t. homogeneous chain decompositions.  
	
	This line of work can easily be extended to block-sequential systems~\cite{goles}, i.e.~dynamical systems exhibiting
	parallel as well as sequential updates.
	Given a connected simple graph $G$ with $n$ vertices,
	let $\mathfrak{p}=S_1 \cdots S_k$ be an \emph{ordered partition} of $V(G)$ into $k$ blocks, i.e., $S_i\subset V(G)$,
	$S_i\cap S_j=\emptyset $ for $i\neq j$, and $\bigcup_{i=1}^k S_i =V(G)$. The
	$S_i$-vertices are updated in a parallel, while for $i<j$, the block $S_i$ is updated before $S_j$.
	If the local functions $f_{v_i}$ are linear, the system is called a \emph{linear} block-sequential dynamical system. 
	The system (and system map) will be denoted as $(G,f,\mathfrak{p})$ or $(G,A,\mathfrak{p})$.
	
	Suppose $S_i=\{v_{i 1}, \ldots, v_{i k_i}\}$. Obviously, $\sum_{i} k_i =n$. We first observe that parallel updating
	the states of the vertices in $S_i$ is equivalent to applying the transformation
	$$
	F_{S_i}=\left(\begin{array}{cccccc}
	1 & 0 & 0 &\cdots &0 & 0\\
	0 & 1 & 0 &\cdots &0 & 0\\
	\vdots & \vdots & \vdots &\ddots &\vdots & \vdots\\
	a_{i_1 1} &a_{i_1 2}  &a_{i_1 3}&\cdots &a_{i_1 n-1} & a_{i_1 n}\\
	\vdots & \vdots & \vdots &\ddots &\vdots & \vdots\\
	a_{i_{k_i} 1}&a_{i_{k_i} 2}&a_{i_{k_i} 3}&\cdots & a_{i_{k_i} n-1} & a_{i_{k_i} n}\\
	\vdots & \vdots & \vdots &\ddots &\vdots\\
	0 & 0 & 0 & \cdots & 1 &0\\
	0 & 0 & 0 & \cdots & 0 &1
	\end{array}\right).
	$$
	Then, if $X_1=(G,f,\pi)X_0$, we have
	\begin{align}
	X_1=F_{S_{k}}F_{S_{{k-1}}} \cdots F_{S_{1}}(X_0).
	\end{align}
	In analogy to Theorem~\ref{2thm1}, we can obtain a closed matrix formula for the composition of these transformations.
	
	Let $A_{\mathfrak{p}}$ be the matrix defined as follows:
	\begin{itemize}
		\item if $v_{i_x}\in S_i$ and $v_{j_y} \in S_j$ are adjacent in $G$, and $i > j$, then $[A_{\mathfrak{p}}]_{i_x j_y}= a_{i_x j_y}$;
		\item and $[A_{\mathfrak{p}}]_{i_x j_y}= 0$, otherwise.
	\end{itemize}
	
	Then we have the following extension of Theorem~\ref{2thm1}:
	
	%%%
	%%%%%%%%%%%%%%%%%%%%%%%%%%%%%%%%%%%%%%%%%%%%%%%%%%%%%%%%%%%%%%%%%%%%%%%%%%%%%%%%%%%%%%%%%%%%%%%%%%%
	%%%
	\begin{theorem}\label{5thm1}
		Let $(G,A,\mathfrak{p})$ be a linear block-sequential dynamical system. Then,
		\begin{align}\label{5t1}
		(G,A,\mathfrak{p})=(I-A_{\mathfrak{p}})^{-1}(A-A_{\mathfrak{p}}).
		\end{align}
	\end{theorem}
	%%%
	%%%%%%%%%%%%%%%%%%%%%%%%%%%%%%%%%%%%%%%%%%%%%%%%%%%%%%%%%%%%%%%%%%%%%%%%%%%%%%%%%%%%%%%%%%%%%%%%%%%
	%%%
	Note, if $V(G)$ is organized in just a single block (i.e.~we have a synchronous system), then $A_{\mathfrak{p}}=0$,
	by construction and the system map equals $A$. If $V(G)$ is organized into $n$ blocks of size one,
	then $A_{\mathfrak{p}}=A_{\pi}$ for some permutation $\pi$ on $V(G)$ and we recover Theorem~\ref{2thm1}.
	
	Theorem~\ref{3thm2} motivates to study the question of whether or not any synchronous system $\mathbb{K}^n
	\rightarrow \mathbb{K}^n$ can be written as an SDS over $\mathbb{K}$. It is easy to construct
	an SDS for such a system $\phi:\mathbb{K}^n
	\rightarrow \mathbb{K}^n$ by doubling the dimension of the phase space as follows:
	
	\begin{itemize}
		\item let the coordinates in the tuples in $\mathbb{K}^n$ be indexed by $v_1,\dots, v_n$;
		\item we construct a complete bipartite graph $K_{n,n}$ having vertices $v_1,\dots,v_n,u_1,\dots,u_n$; 
		\item the local function $f_{v_i}: x_{v_i}\mapsto x_{u_i}$, and $f_{u_i}: x_{u_i}\mapsto \phi([\dots, x_{v_i},\dots])|_i$, i.e., induced by the projection onto the $i$-th coordinate of $\phi$-value;
		\item update the system sequentially following the order $u_1\cdots u_n v_1\cdots v_n$.		
	\end{itemize}
	Let the coordinates of the system states of this SDS be indexed by $v_1,\dots, v_n,u_1,\dots,u_n$. Then, it is not hard to check that the SDS map restricted to the first $n$ coordinates represents the map $\phi$.

	For any finite poset $P$ with $\hat{0}$ and $\hat{1}$, the number $\mu_P(\hat{0},\hat{1})$ encodes key information. Proposition~\ref{P:moeb} establishes
	the computation of the M\"{o}bius function of $P$ by means of an SDS, constructed from $P$. This motivates to ask if
	key $P$-determinants, in particular, $\mu_P(\hat{0},\hat{1})$, have an SDS-interpretation.
	
	Finally, we are interested in studying applications of Theorem~\ref{T:rel} to specific posets, for instance, the
	poset arising from divisibility relations of natural numbers, or root posets of Coxeter systems.

	\section*{Acknowledgements}
	We are grateful to the anonymous referees for their instructive comments and suggestions.
	We also thank Henning Mortveit and Matthew Macauley for their feedback and input on early versions of this paper.
	
	
	

	%\begin{thebibliography}{00}
	%%% Text of bibliographic item
	%
	%\bibitem{}
	%
	%\end{thebibliography}
\end{document}